\documentclass[12pt]{article}
\usepackage{amsfonts}
\usepackage{mathrsfs}
\usepackage{amssymb,amsmath,mathrsfs, amsthm}
\usepackage{palatino,cite}
\usepackage{graphicx}
\usepackage{mathtools}
\usepackage{hyperref,cases,anysize,enumerate}

\numberwithin{equation}{section}

\newtheorem{theorem}{Theorem}[section]
\newtheorem{lemma}{Lemma}[section]
\newtheorem{corollary}{Corollary}[section]
\newtheorem{proposition}{Proposition}[section]
\theoremstyle{definition}
\newtheorem{definition}{Definition}[section]
\newtheorem{remark}{Remark}[section]
\newtheorem{example}{Example}[section]
\newtheorem{question}{Question}[section]
\newtheorem{conjecture}{Conjecture}[section]
\theoremstyle{remark}

\date{}
\begin{document}

\title{Complete complex Finsler metrics and uniform equivalence of the Kobayashi metric}
\author{Jun Nie (jniemath@126.com)\\
Department of Mathematics, Nanchang University\\ Nanchang,  330031, China\\
}

\date{}
\maketitle

\begin{abstract}
  In this paper, first of all, according to Lu's and Zhang's works about the curvature of the Bergman metric on a bounded domain and the properties of the squeezing functions, we observe that Bergman curvatures of the Bergman metric on a bounded strictly pseudoconvex domain with $C^2$-boundary or bounded convex domain are bounded. Applying to the Schwarz lemma from a complete K\"ahler manifold into a complex Finsler manifold, we get that a bounded strictly pseudoconvex domain with $C^2$-boundary or bounded convex domain admits complete strongly pseudoconvex complex Finsler metrics such that their holomorphic sectional curvature is bounded from above by a negative constant. Finally, by the Schwarz lemma from a complete K\"ahler manifold into a complex Finsler manifold, we prove the uniform equivalences of the Kobayashi metric and Carath\'eodory metric on a bounded strongly convex domain with smooth boundary.
\end{abstract}



\section{Introduction}
\noindent

Bounded domains are elementary objects of study in complex analysis. From the Riemann mapping theorem, we know that simply connected open subset in $\mathbb{C}$ which is not the whole complex plane is biholomorphic to the unit disk. In several complex variables, the situation is more complicated and mysterious. To study complex and geometric structures of bounded domains, one may consider that holomorphic mappings from bounded domains to some standard domains such as balls and vice versa. Deng et al. \cite{DGZ1,DGZ2} introduced the notion of squeezing function (see Definition \ref{D-3.1} and \ref{D-3.2}) to study geometric and analytic properties of bounded domains. 

By definition of squeezing function, it is clear that the squeezing function is invariant under biholomorphic transformations. Squeezing function is always positive and bounded above by 1. Now a natural question arises: for arbitrary bounded domains in $\mathbb{C}^n$, how much can we say about the lower bound of the squeezing function? It turns out that the above mentioned squeezing function admits uniform lower and upper bounds on some general classes of bounded domains-\textit{Holomorphic homogeneous regular manifolds} (HHR) or the \textit{Uniformly squeezing domain} (USq), which introduced by Liu et al. \cite{LSY} and by Yeung \cite{YSK} independently.

In 2012, Deng et al. \cite{DGZ1} proved the squeezing function on any bounded domain is continuous. Soon after, they study boundary behaviors of the squeezing functions on some bounded domains. And they proved that the squeezing function of any strictly pseudoconvex domain tends to $1$ near the boundary. What's more, for a bounded strictly pseudoconvex domain with $C^2$-boundary in $\mathbb{C}^n$, it is HHR. In 2016, Kim and Zhang \cite{KZ} proved that the bounded convex domains admit the uniform squeezing property. Furthermore, all bounded convex domains in $\mathbb{C}^n(n \geq 1)$ are HHR. By the above results, we can know that the squeezing functions of bounded strictly pseudoconvex domains with $C^2$-boundary or bounded convex domains in $\mathbb{C}^n$  are bounded from below by a positive constant.

The terminology, HHR or USq, has been introduced in order to the study of completeness and equivalence of the classical invariant metrics, including the Teichm\"uller metric, the Bergman metric, the complete K\"ahler-Einstein metric of negative scalar curvature, the Carath\'eodory metric and the Kobayashi metric. The completeness of these metrics on various types of complex manifolds and the boundness of the holomorphic sectional curvature of the Bergman metric has been an interesting topic ever since 1926 (when the Carath\'eodory metric was first introduced \cite{Caratheodory}). It is known that the Kobayashi metric is the maximum metric among metric satisfying the decreasing property, however, the Carath\'eodory metric is the minimum. It is known that on a bounded, smooth, strictly, pseudoconvex domain in $\mathbb{C}^n$, all classical invariant metrics are uniformly equivalent to each other (see, for example, \cite{BFG,lempert,lempert1,WHS,LSY,DGZ2,YSK} and the references therein).

It has been well-known that for any bounded domain in $\mathbb{C}^n$, the holomorphic sectional curvature of the Bergman metric is always less than $2$ (see \cite{Kobayashi0}). In 2015, Lu \cite{LQK} obtained some lower bounds of the holomorphic sectional curvature and Ricci curvature of the Bergman metric on a bounded domain of $\mathbb{C}^n$. The methods rely on constructing proper minimal functions. We note that the lower bounds Lu obtained tends to $-\infty$ at the point $z$ goes to the boundary $\partial D$. Based on Lu's work, Zhang \cite{ZLY} studied the scalar curvature of the Bergman metric on a bounded domain of $\mathbb{C}^n$. They gave the lower and upper bounds estimates for the Bergman curvatures in terms of the squeezing function originally introduced by Deng et al. \cite{DGZ1}. Combining the above results with Lu's and Zhang's work, the holomorphic sectional curvature of the Bergman metric on a bounded strictly pseudoconvex domain with $C^2$-boundary or bounded convex domain is bounded.

In complex Finsler geometry, there are numerous strongly pseudoconvex complex Finsler metrics in the sense of Abate and Patrizio. In invariant metrics, the Kobayashi metric and the Carath\'eodory metric are complex Finsler metrics. Due to Lempert's result (see Theorem \ref{Lempert} and Remark \ref{R-5.1}), the Kobayashi metric and the Carath\'eodory metric on bounded strongly convex with smooth  boundary are complex Finsler metrics in sense of Abate and Patrizio \cite{abate}. But they are not explicit. Zhong \cite{ZCP} found there are  non-Hermitian quadratic strongly pseudoconvex complex Finsler metrics on a domain of $\mathbb{C}^n$ in the sense of smooth  category. It is very natural to ask the question: $\textit{are these metrics complete ?}$ The following theorem answers this question.
\begin{theorem}(cf. Theorem \ref{MT})\label{T-1.1}
For a suitable choice of positive constant $C$, suppose that $\mathcal{D}$ is a bounded strictly pseudoconvex domain with $C^2$-boundary or bounded convex domain in $\mathbb{C}^n$. Suppose that $\mathcal{D}$ admits a strongly pseudoconvex complex Finsler metric $G:T^{1,0}\mathcal{D}\rightarrow [0,+\infty)$ such that its holomorphic sectional curvature is bounded from above by a negative constant $-K$ and $g_B$ is the Bergman metric on $\mathcal{D}$. Then $H=CG+g_B$ and the Bergman metric $g_B$ are equivalent. What's more, $H=CG+g_B$ is complete.
\end{theorem}

It is well-known that any bounded domains in $\mathbb{C}^n$ are Kobayashi-hyperbolic (see Definition \ref{D-2.3}). What's more, if a domain $\mathcal{D}$ in $\mathbb{C}^n$ is HHR, $\mathcal{D}$ is a complete Kobayashi-hyperbolic. In \cite{WHS}, Wu raised a question as follows.
\begin{question}\label{Q-1.1}\cite[p. 665, Question 1]{WHS}
Does a (complete) Kobayashi-hyperbolic manifold $M$ admit a (complete) Hermitian metric of strongly negative holomorphic sectional curvature, in sense that there exists a negative constant $c$ such that the holomorphic sectional curvature $H$ satisfies $H <c$ ?
\end{question}
\begin{remark}
Cheung \cite{CCK} showed  the result as follows. If $M$ is a compact Kobayashi-hyperbolic manifold such that there exists a surjective holomorphic map $\pi: M \rightarrow Y$ of everywhere maximal rank, such that $Y$ carries a smooth  Hermitian metric of strongly negative holomorphic sectional curvature, and such each fiber $\pi^{-1}(y) (y \in Y)$ also carries a smooth  Hermitian metric of strongly negative holomorphic sectional curvature, then $M$ also admits a smooth  Hermitian metric of strongly negative holomorphic sectional curvature. To and Yeung \cite{TY} showed that the base manifold of an effectively parametrized holomorphic family of compact canonically polarized complex manifolds admits a smooth  invariant complex Finsler metric whose holomorphic sectional curvature is bounded above by a negative constant.
\end{remark}
Recently, Nie and Zhong \cite{NZ1} proved that a bounded domain in $\mathbb{C}^n$ admits a non-Hermitian quadratic strongly pseudoconvex complex Finsler metric such that its holomorphic sectional curvature is bounded from above by a negative constant. Actually, these complex Finsler metrics are strongly convex. Combining Theorem \ref{T-1.1} and the above result, we obtain the following theorem. Our result partly answers Question \ref{Q-1.1}.
\begin{theorem} (cf. Corollary \ref{Coro})\label{T-1.3}
Suppose that $\mathcal{D}$ is a bounded strictly pseudoconvex domain with $C^2$-boundary or bounded convex domain in $\mathbb{C}^n$. Then $\mathcal{D}$ admits a complete strongly pseudoconvex complex Finsler metric $H:T^{1,0}\mathcal{D}\rightarrow [0,+\infty)$ such that its holomorphic sectional curvature is bounded from above by a negative constant. What's more, $\mathcal{D}$ is a complete Kobayashi-hyperbolic.
\end{theorem}

In \cite{NZ1}, Nie and Zhong gave an explicit non-Hermitian quadratic strongly pseudoconvex complex Finsler metric such that its holomorphic sectional curvature is bounded from above by a negative constant. Applying to Theorem \ref{T-1.1}, we construct that the following complex Finsler metric is equivalent to the Bergman metric $g_B$ on $\mathbb{B}^n(\ell)$.
\begin{example}
Let $N:=\mathbb{B}^n(\ell)=\{\|z\|^2<\ell^2\}$ be an open ball in $\mathbb{C}^n(n\geq 2)$ such that $\ell<\frac{1}{b}$ with $b$ an arbitrary positive constant. Let
$$
H(z;v)=C\|v\|^2\exp\Big\{a\|z\|^2+b\frac{|\langle z,v\rangle|^2}{\|v\|^2}\Big\}+g_B(z,v),\quad \forall z\in N,\forall 0\neq v\in T_z^{1,0}N,
$$
where $a,b$ is an arbitrary positive constant,  $C$ is a suitable positive constant, $g_B(z,v)$ is the Bergman metric on $\mathbb{B}^n(\ell)$.
Then $H:T^{1,0}N\rightarrow [0,+\infty)$ is a complete non-Hermitian quadratic strongly pseudoconvex complex Finsler metric with holomorphic sectional curvature bounded from above by a negative constant.
\end{example}

In K\"ahler geometry, Greene and Wu \cite{GW} have posted a remarkable conjecture concerning the uniform equivalences of the Kobayashi-Royden metric. This conjecture is stated below.

\begin{conjecture}\cite[p. 112, Remark (2)]{GW}
Let $(M,h)$ be a simply-connected complete K\"ahler manifold satisfying $-A \leq \text{sectional curvature} \leq -B$ for two positive constants $A$ and $B$. Then, the Kobayashi metric $\mathfrak{K}$ satisfies
$$C^{-1} h(z;v) \leq\mathfrak{K}^2(z;v) \leq C h(z;v), \forall z \in M, \forall v \in T^{1,0}_zM.$$
Here $C>0$ is a constant depending only on $A, B$.
\end{conjecture}
 In \cite{WY}, Wu and Yau confirmed this conjecture in the following theorem. 
 \begin{theorem}\cite[p. 104, Theorem 2]{WY}
 Let $(M,h)$ be a complete K\"ahler manifold whose holomorphic sectional curvature $K_h$ satisfies $-A \leq K_h \leq -B$ for some positive constants $A$ and $B$. Then the Kobayashi metric $\mathfrak{K}$ satisfies
 $$C^{-1} h(z;v) \leq\mathfrak{K}^2(z;v) \leq Ch(z;v), \forall z \in M, \forall v \in T^{1,0}_zM.$$
Here $C>0$ is a constant depending only on $A, B$, and $\dim M$.
 \end{theorem}
As point out in \cite{GW}, it is well known that the left inequality in the conjecture follows from the Ahlfors-Schwarz lemma and the hypothesis of sectional curvature bounded above by a negative constant.  It is very natural to ask the question: \textit{Does the right inequality in the conjecture follow from the Schwarz lemma and the hypothesis of holomorphic sectional curvature ?} Applying to the Schwarz lemma from a complete K\"ahler into a complex Finsler manifold \cite{NZ1} (For more details about Schwarz lemma on complex Finsler manifolds, we refer to \cite{Nie,NZ2,LQZ,wan,shen}), we get the following theorem.
\begin{theorem}(cf. Corollary \ref{C-7.1})\label{T-1.6}
 Suppose that $\mathcal{D}$ is a bounded strongly convex domain with smooth boundary in $\mathbb{C}^n$. And let $(\mathcal{D},h)$ be a complete K\"ahler manifold whose holomorphic sectional curvature $K_h$ satisfies $-A \leq K_h \leq -B$ for some positive constants $A$ and $B$. Then Kobayashi metric and Carath\'eodory metric satisfy
 $$ C^{-1} h(z;v) \leq\mathfrak{K}^2(z;v)=\mathfrak{C}^2(z;v) \leq Ch(z;v), \forall z \in \mathcal{D}, \forall v \in T^{1,0}_z\mathcal{D}.$$
 Here $C>0$ is depending only $A, B$, and independent of complex dimension $n$ of complex manifold $\mathcal{D}$.
 \end{theorem}

 If the Carath\'eodory pseudometric on a complex manifold $M$ is a smooth  strongly pseudoconvex complex Finsler metric, we have the following theorem.
 \begin{theorem}(cf. Theorem \ref{T-7.2})\label{T-1.7}
For a suitable choice of positive constant $C$, suppose that $(M,h)$ is a complete K\"ahler manifold whose holomorphic sectional curvature $K_h$ satisfies $-A \leq K_h \leq B$ for some positive constants $A$ and $B$. If the Carath\'eodory metric $\mathfrak{C}^2$ on complex manifold $M$ is a smooth  strongly pseudoconvex complex Finsler metric. Then $H=C\mathfrak{C}^2+h$ and the K\"ahler metric $h$ are equivalent. What's more, $M$  is a complete Kobayashi-hyperbolic manifold.
\end{theorem}
\section{Terminology and preliminaries}
\noindent

In this section, we recall the following some basic definitions, facts and propositions in this paper. For more details, we refer to \cite{abate,kobayashi,chenb,CK}.

\subsection{Complex Finsler geometry}
\noindent

Let $M$ be a complex manifold of complex dimension $n$. Let $\{z^1,\cdots,z^n\}$ be a set of local complex coordinates, and let $\{\frac{\partial}{\partial z^{\alpha}}\}_{1 \leq \alpha \leq n}$ be the corresponding natural frame of $T^{1,0}M$. So any non-zero element in $\widetilde{M}=T^{1,0}M \setminus \{\text{zero section}\}$ can be written as
$$v=v^{\alpha}\frac{\partial}{\partial z^{\alpha}} \in \widetilde{M},$$
where we adopt the summation convention of Einstein. In this way, one gets a local coordinate system on the complex manifold $\widetilde{M}$:
$$(z;v)=(z^1,\cdots,z^n;v^1,\cdots,v^n).$$

\begin{definition} \cite{abate}
A complex Finsler metric $G:=F^2$ on a complex manifold $M$ is a continuous function $G:T^{1,0}M \rightarrow [0,+\infty)$ satisfying $G$ is smooth  on $\tilde{M}:=T^{1,0}M\setminus\{\mbox{zero section}\}$ and
$$G(z;\zeta v)=|\zeta|^2G(z;v)$$
 for all $(z;v) \in T^{1,0}M$ and $\zeta \in \mathbb{C}$.
A complex Finsler metric $G$ is called strongly pseudoconvex if the Levi matrix
\begin{equation*}
(G_{\alpha \overline{\beta}})= \Big(\frac{\partial^2 G}{\partial v^{\alpha}\partial \overline{v}^{\beta}}\Big)
\end{equation*}
is positive definite on $\widetilde{M}$.
\end{definition}
Any $C^\infty$ Hermitian metric on a complex manifold $M$ is naturally a strongly pseudoconvex complex Finsler metric. Conversely,
if a complex Finsler metric $G$ on a complex manifold $M$ is $C^\infty$ over the whole holomorphic tangent bundle $T^{1,0}M$, then it is necessary a $C^\infty$ Hermitian metric. For this reason, in general the non-trivial (non-Hermitian quadratic) examples of complex Finsler metrics are only required to be smooth  over the slit holomorphic tangent bundle $\widetilde{M}$. 

If $\lambda=\lambda(z) dz\otimes d{\bar{z}}, \lambda(z)>0$ is a Hermitian metric on a Riemann surface $R$. Then the Gaussian curvature $K(\lambda)(z)$ of $\lambda$ at the point $z$ is given by
$$K(\lambda)(z)=- \frac{1}{\lambda(z)}\frac{\partial^2 \log \lambda}{\partial z \partial \bar{z}}.$$
\begin{proposition}\label{L-2.2}
 For a positive constant $C>0$, let $\lambda$ be a Hermitian metric on a Riemann surface $R$. Then
 $$K(C\lambda)(z)=\frac{1}{C}K(\lambda)(z).$$
\end{proposition}
\begin{proof}
\begin{align*}\begin{split}
K(C\lambda)(z)&=-\frac{1}{C\lambda(z)}\frac{\partial^2 \log (C\lambda(z))}{\partial z \partial \bar{z}}\\
&=-\frac{1}{C\lambda(z)}\frac{\partial^2 (\log\lambda(z)+ \log C)}{\partial z \partial \bar{z}}\\
&=-\frac{1}{C}\frac{1}{\lambda(z)}\frac{\partial^2 \log \lambda(z)}{\partial z \partial \bar{z}}\\
&=\frac{1}{C}K(\lambda)(z).
\end{split}\end{align*}
\end{proof}

The following proposition was actually outlined and used in Wong and Wu \cite[Lemma 4.3, p. 430]{WW}. We derive the proposition needed for Theorem \ref{MT} and give a brief proof here.
\begin{proposition} \label{L-2.1}
Suppose that $\lambda, \mu$ are Hermitian metrics on a Riemann surface $R$. Then
$$ K(\lambda+\mu)(z) \leq \frac{\lambda^2}{(\lambda+\mu)^2}K(\lambda)(z)+\frac{\mu^2}{(\lambda+\mu)^2}K(\mu)(z),$$
where $K(\lambda)(z), K(\mu)(z), K(\lambda+\mu)(z)$ are the Gaussian curvature of the Hermitian metric $\lambda, \mu, \lambda+\mu$ at the point $z$, respectively. Moreover,
$$K(\lambda+\mu)(z) \leq K(\lambda)(z)+K(\mu)(z).$$
\end{proposition}
\begin{proof}
By direct calculation, we have
\begin{align*}\begin{split}
\frac{\partial^2 \log (\lambda+\mu)}{\partial z \partial \bar{z}}&=\frac{\partial }{\partial z}\Big(\frac{\frac{\partial \lambda}{\partial \bar{z}}+\frac{\partial \mu}{\partial \bar{z}}}{\lambda+\mu}\Big)\\
&=\frac{\lambda\frac{\partial^2\lambda}{\partial z \partial \bar{z}}-\frac{\partial \lambda}{\partial z}\frac{\partial \lambda}{\partial \bar{z}} +\mu\frac{\partial^2\mu}{\partial z\partial \bar{z}}-\frac{\partial \mu}{\partial z}\frac{\partial \mu}{\partial \bar{z}}+\lambda\frac{\partial^2\mu}{\partial z \partial \bar{z}}+ \mu\frac{\partial^2\lambda}{\partial z \partial \bar{z}}-\frac{\partial \lambda}{\partial z}\frac{\partial \mu}{\partial \bar{z}}-\frac{\partial \mu}{\partial z}\frac{\partial \lambda}{\partial \bar{z}}}{(\lambda+\mu)^2}\\
&=\frac{\lambda^2\frac{\partial^2 \log \lambda}{\partial z \partial \bar{z}}+\mu^2\frac{\partial^2 \log \mu}{\partial z \partial \bar{z}}+\frac{\lambda}{\mu}(\mu^2\frac{\partial^2 \log \mu}{\partial z \partial \bar{z}}+\frac{\partial \mu}{\partial z}\frac{\partial \mu}{\partial \bar{z}})}{(\lambda+\mu)^2}\\
&+\frac{\frac{\mu}{\lambda}(\lambda^2\frac{\partial^2 \log \lambda}{\partial z \partial \bar{z}}+\frac{\partial \lambda}{\partial z}\frac{\partial \lambda}{\partial \bar{z}}) -\frac{\partial \lambda}{\partial z}\frac{\partial \mu}{\partial \bar{z}}-\frac{\partial \mu}{\partial z}\frac{\partial \lambda}{\partial \bar{z}}}{(\lambda+\mu)^2}\\
&=\frac{\lambda^2\frac{\partial^2 \log \lambda}{\partial z \partial \bar{z}}+\mu^2\frac{\partial^2 \log \mu}{\partial z \partial \bar{z}}+\lambda\mu(\frac{\partial^2 \log \mu}{\partial z \partial \bar{z}}+\frac{\partial \log \mu}{\partial z}\frac{\partial \log \mu}{\partial \bar{z}})}{(\lambda+\mu)^2}\\
&+\frac{\lambda\mu(\frac{\partial^2 \log \lambda}{\partial z \partial \bar{z}}+\frac{\partial \log \lambda}{\partial z}\frac{\partial \log \lambda}{\partial \bar{z}})-\lambda\mu\frac{\partial \log \lambda}{\partial z}\frac{\partial \log \mu}{\partial \bar{z}}-\lambda\mu\frac{\partial \log \mu}{\partial z}\frac{\partial \log \lambda}{\partial \bar{z}}}{(\lambda+\mu)^2}\\
&=\frac{\lambda^2\frac{\partial^2 \log \lambda}{\partial z \partial \bar{z}}+\mu^2\frac{\partial^2 \log \mu}{\partial z \partial \bar{z}}+\lambda\mu\frac{\partial^2 \log \mu}{\partial z \partial \bar{z}}+\lambda\mu\frac{\partial^2 \log \lambda}{\partial z \partial \bar{z}}}{(\lambda+\mu)^2}\\
&+\frac{\lambda\mu(\frac{ \partial \log \lambda}{\partial z}-\frac{\partial \log \mu}{\partial z})(\frac{ \partial \log \lambda}{\partial \bar{z}}-\frac{\partial \log \mu}{\partial \bar{z}})}{(\lambda+\mu)^2}\\
&=\frac{\lambda(\lambda+\mu)\frac{\partial^2 \log \lambda}{\partial z \partial \bar{z}}+\mu(\lambda+\mu)\frac{\partial^2 \log \mu}{\partial z \partial \bar{z}}+\lambda\mu|\frac{ \partial \log \lambda}{\partial z}-\frac{\partial \log \mu}{\partial z}|^2}{(\lambda+\mu)^2}
\end{split}\end{align*}
Therefore we have
\begin{align*}\begin{split}
&-\frac{\partial^2 \log (\lambda+\mu)}{\partial z \partial \bar{z}} \leq -\frac{\lambda}{\lambda+\mu}\frac{\partial^2 \log \lambda}{\partial z \partial \bar{z}}-\frac{\mu}{\lambda+\mu}\frac{\partial^2 \log \mu}{\partial z \partial \bar{z}},
\end{split}\end{align*}
By the definition of Gaussian curvature, we obtain
$$K(\lambda+\mu)(z) \leq \frac{\lambda^2}{(\lambda+\mu)^2}K(\lambda)(z)+\frac{\mu^2}{(\lambda+\mu)^2}K(\mu)(z)\leq K(\lambda)(z)+K(\mu)(z).$$
\end{proof}

Let $(M,G)$ be a strongly pseudoconvex complex Finsler manifold $M$, and take $v \in \widetilde{M}$. Then the holomorphic sectional curvature $K_G(v)$ of $G$ along a non zero tangent vector $v$ is given by
\begin{equation*}
K_G(v)=\frac{2}{G(v)^2}\langle\Omega(\chi,\bar{\chi})\chi,\chi\rangle_v.
\end{equation*}
where $\chi=v^\alpha\delta_\alpha$ is the complex radial horizontal vector field and $\Omega$ is the curvature tensor of the Chern-Finsler connection associated to $(M,G)$.

Abate and Patrizio found an phenomenon that the holomorphic sectional curvature of a complex Finsler metric $G$ is the supremum of the Gaussian curvature of the induced metric through a family of holomorphic maps.

\begin{proposition}\label{P-2.1} \cite[p. 110, Corollary 2.5.4]{abate}
Suppose that $G:T^{1,0}M \rightarrow [0,+\infty)$ a strongly pseudoconvex complex Finsler metric on a complex manifold $M$, take $z \in M$ and $0 \neq v \in T_z^{1,0}M$. Then
$$K_G(v)=\sup\{K(\varphi^*G)(0)\},$$
where the supremum is taken with respect to the family of all holomorphic maps $\varphi: \mathbb{D} \rightarrow M$ with $\varphi(0)=z$ and $\varphi'(0)=\lambda v$ for some $\lambda \in \mathbb{C}\backslash \{0\}$; $K(\varphi^*G)(0)$ is the Gaussian curvature of $(\mathbb{D},\varphi^*G)$ at the point $0$.
\end{proposition}

\subsection{Invariant metrics}
\noindent

The Bergman pseudometric $g_B$ on a complex manifold $M$ of complex dimension $n$ is a K\"ahler pseudometric with local potential given by the coefficients of the Bergman kernel $K(x, x)$. Let $f$ be a $L^2$-holomorphic $n$-form on $M$. In terms of local coordinates $\left(z_1, \ldots, z_n\right)$ on a coordinate chart $U$, let $e_{K_M}=d z^1 \wedge \ldots \wedge d z^n$ be a local basis of the canonical line bundle $K_M$ on $U$. We can write $f$ as $f_U e_{K_M}$ on $U$. Let $f_i$ be an orthonormal basis of $L^2$-sections in $H_{(2)}^0\left(M, K_M\right)$. Note that from conformality, the choice is independent of the metric on $M$. The Bergman kernel is given by $K(x, x)=\sum_i f_i \wedge \overline{f_i}$. Let $K_U(x, x)=\sum_i f_{U, i} \overline{f_{U, i}}$ be the coefficient of $K(x, x)$ in terms of the local coordinates. The Bergman metric is given by the K\"ahler form
$$
\omega_B=\sqrt{-1} \partial \bar{\partial} \log K_U(x, x)=\sqrt{-1} \frac{1}{K_U(x, x)^2} \sum_{i<j}\left(f_i \partial f_j-f_j \partial f_i\right) \wedge \overline{\left(f_i \partial f_j-f_j \partial f_i\right)},
$$
which is clearly independent of the choice of a basis and $U$. As the Bergman kernel is independent of basis, for each fixed point $x \in M$,
$$
K_U(x, x)=\sup _{f \in H_{(2)}^0\left(M, K_M\right),\|f\|=1}\left|f_U(x)\right|^2,
$$
where $\|\cdot\|$ stands for the $L^2$-norm. We may assume that $\sup _{f \in H_{(2)}^0\left(M, K_M\right),\|f\|=1}\left|f_U(x)\right|$ is realized by $f_x \in H_{(2)}^0(M, K)$ with $\left\|f_x\right\|=1$ so that $K_U(x, x)=\left|f_{x, U}(x)\right|^2$.

For a complex manifold, there are natural and intrinsic complex Finsler pseudometric, i.e., the Kobayashi pseudometric and the Carath\'eodory pseudometric. It is known that the Kobayashi pseudometric is the maximum pseudometric among pseudometric satisfying the decreasing property. However, the Carath\'eodory pseudometric is the minimum. Let us recall some definitions.

\begin{definition}\cite[p. 86, Section 3.5]{kobayashi}
Let $M$ be a complex manifold. For any $(z,v) \in T^{1,0}M$, then the \textit{Kobayashi pseudometric} $\mathfrak{K}_M: T^{1,0}M \rightarrow [0,+\infty)$ of $M$ is given by
$$\mathfrak{K}_M(z,v)=\inf_{R>0} \frac{1}{R},$$
where $R$ ranges over all positive numbers for which there is a $\phi \in \text{Hol}(\mathbb{D}_R,M)$ with $\phi(0)=z$, and $\phi_*(\frac{\partial}{\partial z}|_{z=0})=v$. Here $\text{Hol}(X,Y)$ denotes the set of holomorphic maps from $X$ to $Y$, and
$\mathbb{D}_R\equiv \{z \in \mathbb{C}; |z| <R\}$ and $\mathbb{D}\equiv \mathbb{D}_1.$
\end{definition}
Equivalently, one can verify that, for each $(z,v) \in T^{1,0}M$,
\begin{align*}\begin{split}
\mathfrak{K}_M(z,v)&=\inf \{ |V|_{\mathcal{P}};V \in T^{1,0}\mathbb{D}, \exists~\phi \in \text{Hol}(\mathbb{D}, M)~\text{with}~\phi_*(V)=v\}\\
&=\inf \{ |V|_{\mathbb{C}};V \in T^{1,0}_0\mathbb{D}, \exists~\phi \in \text{Hol}(\mathbb{D}, M)~\text{such that}~\phi(0)=x,~\phi_*(V)=v\}.
\end{split}\end{align*}
Where $|\cdot|_{\mathcal{P}}$ and $|\cdot|_{\mathbb{C}}$ are, respectively, the norms with respect to the Poincar\'e metric $\mathcal{P}=(1-|z|^2)^{-2} dz \otimes d\overline{z}$ and Euclidean metric ${E}=dz \otimes d\overline{z}$.

\begin{definition}\cite[p. 84, Definition 2.3.4]{abate}
Suppose that $M$ is a complex manifold, then the \textit{Carath\'eodory pseudometric} $\mathfrak{C}_M: T^{1,0}M \rightarrow [0, \infty)$ of $M$ is given by:
$$\mathfrak{C}(z,v)=\sup \{ |df_z(v)|_{\mathbb{C}}; \exists f \in \text{Hol}(M,\mathbb{D}) ~\text{with}~ f(0)=z\}, \forall (z,v) \in T^{1,0}M.$$
\end{definition}

If we consider bounded domains in $\mathbb{C}^n$, both Kobayashi pseudometric and Carath\'eodory pseudometric are non-degenerate complex Finsler metrics. In 1980s, Lempert \cite{lempert,lempert1} proved the following theorem.
\begin{theorem}\cite[ also see p. 220, Theorem 4.8.13]{kobayashi}\label{Lempert}
If $\mathcal{D}$ is a bounded strongly convex domain with $C^k$-boundary  in $\mathbb{C}^n$. Then the Kobayashi metric $\mathfrak{K}_{\mathcal{D}}$ and Carath\'eodory metric $\mathfrak{C}_{\mathcal{D}}$ are $C^k$-strongly pseudoconvex complex Finsler metrics such that
$$\mathfrak{K}_{\mathcal{D}}=\mathfrak{C}_{\mathcal{D}}.$$
\end{theorem}
\begin{remark}\cite[see p. 341-343]{lempert1}\label{R-5.1}
 As long as $\mathcal{D}$ is a bounded strongly convex domain with smooth boundary in $\mathbb{C}^n$, the differentiability condition of two intrinsic metrics is satisfied in sense of Abate and Patrizio in this paper.
\end{remark}
In \cite{WB}, Wong proved the holomorphic sectional curvature of the Kobayashi metric $\mathfrak{K}_M$ is greater than or equal to $-4$. Aslo, he proved the holomorphic sectional curvature of the Carath\'eodory metric $\mathfrak{C}_M$ is less than or equal to $-4$. By Theorem \ref{Lempert}, we know that the Kobayashi metric on a bounded strongly convex domain with smooth boundary is complex Finsler metric in sense of Abate and Patrizio whose holomorphic sectional curvature is constant and equal to $-4$.
\begin{definition}\cite[p. 291, Definition 1.1]{Spiro}\label{D-5.1}
A complex manifold $M$ is a \textit{Lempert manifold} if

$(1)$ the Kobayashi pseudometric, denoted by $\mathfrak{K}_M$, (is the infinitesimal form of the Kobayahsi pseudodistance.) is a strongly pseudoconvex complex Finsler metric, that is:

$a)$ it is smooth  function on $\widetilde{M}$ with values $\mathbb{R}^+$;

$b)$ $\mathfrak{K}_M(\lambda v)=|\lambda|\mathfrak{K}_M(v)$ for any $\lambda \in \mathbb{C}^*$ and $v \in \widetilde{M}$;

$c)$ at any point $z \in M$ the hypersurface $S_z= \{v \in T^{1,0}_zM: \mathfrak{K}_M(v)=1\}$ is strongly pseudoconvex in $T^{1,0}_zM$;

$(2)$ for any non-vanishing complex vector $ \omega \in T^{\mathbb{C}}_zM \subset T^{\mathbb{C}}M$, there exists a complex curve $\gamma_\omega: U \subset \mathbb{C} \rightarrow M$, such that $\gamma_\omega(0)=z, \gamma_\omega'(0)=\omega$ and $\gamma(U)$ is a totally geodesic submanifold of $M$;

$(3)$ the metric, which is induced by $\mathfrak{K}_M$ on the totally geodesic complex curve $\gamma_\omega(U)$, is K\"ahler and with constant holomorphic sectional curvature equal to $-4$;

$(4)$ the (finite) Kobayashi distance $d_k$, determined by $\mathfrak{K}_M$, is complete and the exponential map $exp: T_z^{1,0}M \rightarrow M$ is a diffeomorphism for any $z \in M$.
\end{definition}
\begin{remark} \cite[p. 292, Theorem 1.3]{Spiro}\label{R-5.2}
In 1990, Faran\cite{FJJ} proved a complex manifold $M$ is Lempert manifold if and only if it admits a strongly pseudoconvex complex Finsler metric $G$, which verifies $(2), (3)$, and $(4)$ of Definition \ref{D-5.1}. In this case $F$ coincides with the Kobayashi metric $\mathfrak{K}_M$. By Proposition \ref{P-2.1}, we know that the holomorphic sectional curvature of Lempert manifolds is equal to $-4$.
\end{remark}

An immediate interest for Lempert manifolds comes from the well-known result of Lempert on the Kobayashi metric of strongly convex domains in $\mathbb{C}^n$.
\begin{theorem} \cite[p. 292, Theorem 1.2]{Spiro}\label{T-5.2}
If $M$ is a  bounded strongly convex domain with smooth boundary in $\mathbb{C}^n$, then $M$ is a Lempert manifold.
\end{theorem}

\subsection{Kobayashi invariant distance}
\noindent

Now let us discuss the widely used \textit{Kobayashi distance}. Denote by  $(\mathbb{D}, \mathcal{P})$ a unit disc in $\mathbb{C}$ equipped with the Poincar\'e metric
$$
\mathcal{P}=\frac{d z \otimes d \bar{z}}{\left(1-|z|^2\right)^2}
$$
which is complete K\"ahler and has constant curvature -4. The distance function of $\mathcal{P}$ is given by
$$
d_\mathcal{P}(0, z)=\frac{1}{2} \ln \left(\frac{1+|z|}{1-|z|}\right).
$$

Given $p, q \in M$, let us consider a finite set of holomorphic maps $f_i: \mathbb{D} \rightarrow M, 1 \leq i \leq l$, such that there are points $p_1, \cdots, p_l \in \mathbb{D}$ with
$$
f_1(0)=p, f_1\left(p_1\right)=f_2(0), f_2\left(p_2\right)=f_3(0), \ldots, f_{l-1}\left(p_{l-1}\right)=f_l(0), f_l\left(p_l\right)=q.
$$
We  call such a set a \textit{ chain of holomorphic discs} from $p$ to $q$. The \textit{length} of the chain is defined by $\sum_{i=1}^l d_\mathcal{P}\left(0, p_i\right)$.

The \textit{Kobayashi distance} between $p$ and $q$, denoted by $d_k\left(p, q\right)$, is defined by the infimum of the length of any (finite) chain of holomorphic discs from $p$ to $q$. Of course it is always a finite nonnegative number, and satisfies the triangle inequality hence is a pseudo distance.
\begin{definition}\cite[p. 233, Definition 9.12]{ZFY}\label{D-2.3}
A complex manifold $M$ is said to be \textit{Kobayashi hyperbolic} if $d_k$ is a distance function, i.e., if $d_k(p,q)>0$ whenever $p \neq q \in M$.
\end{definition}
\subsection{ Curvature of Bergman metric on a bounded domain }
\noindent

Suppose that $\mathcal{D}$ is a bounded domain. For any $z \in \mathcal{D}$, let $B(z,r(z))\subset \mathcal{D} \subset B(z,R(z))$ denote the interior and exterior with the center $z$ and the radius $r(z)$ and $R(z)$, respectively. Denote by $Sec_{\mathcal{D}}(z,v)$ and $Ric_{\mathcal{D}}(z,v)$ the holomorphic sectional curvature and the Ricci curvature with respect to the Bergman metric at the point $z$ in the direction $v$. Denote by $Scal_{\mathcal{D}}(z)$ the scalar curvature of the Bergman metric at the point $z$.
\begin{definition}\cite[p. 629, Definition 7.2]{LSY}\cite[p. 548, Definition 1]{YSK}\label{D-3.1}
A complex manifold $M$ of complex dimension $n$ is called uniformly squeezing (USq) or equivalently holomorphic homogeneous regular (HHR), if there are uniform positive constants $r$ and $R$ such that for any point $p \in M$, there is a holomorphic map $f_p: M\rightarrow \mathbb{C}^n$ which satisfies

i). $f_p(p)=0$;

ii). $f_p: M \rightarrow f_p(M)$;

iii). $B^{n}(0,r) \subset f_p(M) \subset B^n(0,R)$, where $B^{n}(0,r)$ and $B^n(0,R)$ are Euclidean balls with center $0$ in $\mathbb{C}^n$.
\end{definition}
In \cite{DGZ1}, Deng et al. introduced the concept of squeezing function in order to study analytic and geometric properties for HHR or USq domains.
\begin{definition}\cite[p. 2679, Definition 1.1]{DGZ2}\label{D-3.2}
Let $\mathcal{D}$ be a bounded domain in $\mathbb{C}^n$. For $z\in \mathcal{D}$ and an open holomorphic embedding $f: \mathcal{D}\rightarrow B^n=B^n(0,1)$ with $f(z)=0$, we define
$$s_{\mathcal{D}}(z,f)=\sup\{r| B^n(0,r) \subset f(\mathcal{D})\}\quad  \text{and} \quad s_{\mathcal{D}}(z)=\sup_f{s_{\mathcal{D}}(z,f)},$$
where the supremum is taken over all holomorphic embeddings $f:\mathcal{D} \rightarrow B^{n}$ with $f(z)=0$, $B^n$ is the unit ball in $\mathbb{C}^n$, and $B^n(0,r)$ is the ball in $\mathbb{C}^n$ with center $0$ and radius $r$. As $z$ varies, we get a function $s_{\mathcal{D}}$ on $\mathcal{D}$, which is called the \textit{squeezing function} of $\mathcal{D}$.
\end{definition}
\begin{remark}\label{R-3.2}
If the squeezing constant $ \hat{s}_{\mathcal{D}}$ for $\mathcal{D}:=\inf_{z \in \mathcal{D}}s_{\mathcal{D}}(z)$ is a positive constant,  $\mathcal{D}$ is called \textit{holomorphic homogeneous regular} (HHR), or equivalently \textit{uniformly squeezing} (USq).
\end{remark}

For each point $z \in \mathcal{D}$, Deng et al. showed that there are an extremal holomorphic embedding $f:\mathcal{D} \hookrightarrow B^n$ with $f(z)=0$ and $B(0,s_{\mathcal{D}})\subset f(\mathcal{D})$ (see Theorem 2.1 in \cite{DGZ1}).  Making use of the squeezing function, Zhang \cite{ZLY} got the following theorem.

\begin{theorem}\cite[p. 1150, Theorem 1.1]{ZLY}\label{Bound}
For any point $z\in \mathcal{D}$ and $v \in T^{1,0}_z\mathcal{D}\backslash\{0\}\cong\mathbb{C}^n\backslash\{0\}$, then
\begin{align*}
2-2\frac{n+2}{n+1}s_{\mathcal{D}}^{-4n} \leq Sec_{\mathcal{D}}(z,v) \leq 2-2\frac{n+2}{n+1}s_{\mathcal{D}}^{4n},
\end{align*}
\begin{align*}
(n+1)-2(n+2)s_{\mathcal{D}}^{-2n} \leq Ric_{\mathcal{D}}(z,v) \leq (n+1)-(n+2)s_{\mathcal{D}}^{2n},
\end{align*}
\begin{align*}
n(n+1)-n(n+2)s_{\mathcal{D}}^{-2n} \leq Scal_{\mathcal{D}}(z) \leq n(n+1)-n(n+2)s_{\mathcal{D}}^{2n}.
\end{align*}
\end{theorem}

In \cite{YSK}, Yeung study uniform squeezing property on a bounded domain. And he got the following estimate on the growth of Bergman kernel.
\begin{theorem}\cite[p. 550, Corollary 3]{YSK}
Let $\mathcal{D}$ be a bounded domain with the uniform squeezing property. Denote by $d=d(z, \partial \mathcal{D})$ the Euclidean distance of $z \in \mathcal{D}$ from the boundary $\partial \mathcal{D}$ of $\mathcal{D}$. Then $K(z,z) \geq \frac{c}{d^2(-\log d)^2}$ for some constant $c$.
\end{theorem}
According to the fact that any bounded strictly pseudoconvex domains with $C^2$-boundary or  bounded convex domains in $\mathbb{C}^n$ are HHR, we have the following corollary.
\begin{corollary}
Suppose that $\mathcal{D}$ is a bounded strictly pseudoconvex domain or  bounded convex domain with $C^2$-boundary in $\mathbb{C}^n$. Denote by $d=d(z, \partial \mathcal{D})$ the Euclidean distance of $z \in \mathcal{D}$ from the boundary $\partial \mathcal{D}$ of $\mathcal{D}$. Then $K(z,z) \geq \frac{c}{d^2(-\log d)^2}$ for some constant $c$.
\end{corollary}

\subsection{Holomorphic sectional curvature and sectional curvature on a K\"ahler manifold}
\noindent

 In this section, we introduce some relationship between holomorphic sectional curvature and sectional curvature. In K\"ahler geometry, we can easily know that holomorphic sectional curvature is 'dominated' by the sectional curvature \cite{ZFY}. It is very natural to ask the question: \textit{If holomorphic sectional curvature of a K\"ahler manifold is bounded , can we deduce that its sectional curvature is bounded}? Now, we give affirmative answer to above question. Firstly, we can refer to an exercise  in \cite[p. 189, Exercise 17]{ZFY}. In this paper, we don't need such a strong theorem. Hence, we introduce and prove the following theorem.
\begin{theorem}\label{sectional}
Suppose that $(M,h)$ is a K\"ahler manifold whose holomorphic sectional curvature is bounded, then its sectional curvature is also bounded.
\end{theorem}
\begin{proof}
From the definitions of holomorphic sectional curvature and sectional curvature, we know that curvatures are independent of the length of vectors. Without loss of generality, we assume that $X,Y \in T^{1,0}M$ are unit vectors, i.e., $h(X)=h(Y)=1$. We denote Riemannian curvature tensor of a K\"ahler manifold $(M,h)$ by $R$.

From properties of curvature of K\"ahler manifold $(M,h)$, we have
\begin{align}\begin{split}\label{EQ-4.2}
 &R(X+Y, \overline{X+Y}, X+Y, \overline{X+Y})+R(X-Y, \overline{X-Y}, X-Y, \overline{X-Y})\\
 &=2R(X, \overline{X}, X, \overline{X})+2R(Y, \overline{Y}, Y, \overline{Y})+8R(X, \overline{X}, Y, \overline{Y})\\
 &+2R(X, \overline{Y}, X, \overline{Y})+2R(Y, \overline{X}, Y, \overline{X}).
 \end{split}\end{align}
Similarly,
\begin{align}\begin{split}\label{EQ-4.3}
 &R(X+iY, \overline{X+iY}, X+iY, \overline{X+iY})+R(X-iY, \overline{X-iY}, X-iY, \overline{X-iY})\\
 &=2R(X, \overline{X}, X, \overline{X})+2R(Y, \overline{Y}, Y, \overline{Y})+8R(X, \overline{X}, Y, \overline{Y})\\
 &-2R(X, \overline{Y}, X, \overline{Y})-2R(Y, \overline{X}, Y, \overline{X}).
 \end{split}\end{align}
Combining ~\eqref{EQ-4.2} with ~\eqref{EQ-4.3}, we have
\begin{align}\begin{split}\label{EQ-4.4}
R(X, \overline{X}, Y, \overline{Y})&=\frac{1}{16}\Big(R(X+Y, \overline{X+Y}, X+Y, \overline{X+Y})+R(X-Y, \overline{X-Y}, X-Y, \overline{X-Y})\\
&\quad+R(X+iY, \overline{X+iY}, X+iY, \overline{X+iY})+R(X-iY, \overline{X-iY}, X-iY, \overline{X-iY})\\
&\quad-4R(X, \overline{X}, X, \overline{X})-4R(Y, \overline{Y}, Y, \overline{Y})\Big);
\end{split} \end{align}
 \begin{align}\begin{split}\label{EQ-4.5}
 R(X, \overline{Y}, X, \overline{Y})+R(Y, \overline{X}, Y, \overline{X})&=\frac{1}{4}\Big(R(X+Y, \overline{X+Y}, X+Y, \overline{X+Y})\\
 &+R(X-Y, \overline{X-Y}, X-Y, \overline{X-Y})\\
 &-R(X+iY, \overline{X+iY}, X+iY, \overline{X+iY})\\
 &-R(X-iY, \overline{X-iY}, X-iY, \overline{X-iY})\Big).
 \end{split} \end{align}
We can easily know
$$h(X+Y,X+Y)=h(X,X)+h(Y,Y)+ \mbox{Re}\{h(X,Y)+h(Y,X)\} \leq 2(\sqrt{h(X,X)}+\sqrt{h(Y,Y)})^2=8.$$
Hence $h(X+Y,X+Y)$ is not more than some constant. For the convenience of proving, we may assume $X,Y$ are orthonormal unit vector, then we have
 $$h(X+Y)=h(X)+h(Y)=2.$$
Since the holomorphic sectional curvature of the K\"ahler manifold $(M,h)$ is bounded, we suppose that
$$|R(X, \overline{X}, X, \overline{X})| \leq C,$$
where $C$ is a constant.

By the equalities \eqref{EQ-4.4} and \eqref{EQ-4.5}, we get
\begin{align}\begin{split}\label{EQ-4.6}
|R(X, \overline{X}, Y, \overline{Y})|\leq \frac{1}{16}(8+8+8+8+2+2)C=\frac{5}{2}C,
\end{split} \end{align}
\begin{align}\begin{split}\label{EQ-4.7}
| R(X, \overline{Y}, X, \overline{Y})+R(Y, \overline{X}, Y, \overline{X})|\leq \frac{1}{4}(8+8+8+8)C=8C,
\end{split} \end{align}
From properties of curvature of the K\"ahler manifold $(M,h)$, we obtain
$$R(X,Y,\cdot,\cdot)=R(\overline{X},\overline{Y},\cdot,\cdot)=R(\cdot,\cdot,X,Y)=R(\cdot,\cdot,\overline{X},\overline{Y})=0.$$
Denoted by $u_1=X+\overline{X},u_2=Y+\overline{Y}$, we have
\begin{align}\begin{split}\label{EQ-4.8}
R(u_1,u_2,u_2,u_1)&=R(X+\overline{X},Y+\overline{Y},Y+\overline{Y},X+\overline{X})\\
&= R(X, \overline{Y}, Y, \overline{X})+R(X, \overline{Y}, \overline{Y}, X)\\
&+R(\overline{X}, Y, Y, \overline{X})+R(\overline{X}, Y, \overline{Y}, X)\\
&=2R(X, \overline{X},Y, \overline{Y})-R(X, \overline{Y},X, \overline{Y})-R(Y, \overline{X},Y, \overline{X}).
\end{split} \end{align}
Substituting \eqref{EQ-4.6}-\eqref{EQ-4.7} into \eqref{EQ-4.8}, we have
$$|R(u_1,u_2,u_2,u_1)| \leq 13C.$$
Therefore, the sectional curvature of the K\"ahler manifold $(M,h)$ is also bounded.
\end{proof}

 \section{Strongly negative holomorphic sectional curvature of complex Finsler and Kobayashi-hyperbolic}
 \noindent


 In this section, we will prove the Theorem \ref{T-1.1} (i.e. Theorem \ref{MT}) and \ref{T-1.3} (i.e. Corollary \ref{Coro}). Before giving the proofs of our main results, we need some results as follows. The first one is the Schwarz lemma from a complete K\"ahler manifold into a complex Finsler manifold.
 \begin{theorem}\cite[p. 1662, Theorem 1.1 ]{NZ1} \label{SL}
 Suppose that $(M,ds_M^2)$ is a complete K\"ahler manifold with holomorphic sectional curvature bounded from below by a constant $K_1$ and sectional curvature bounded from below, while $(N,G)$ is a strongly pseudoconvex complex Finsler manifold with holomorphic sectional curvature of the Chern-Finsler connection bounded from above by a constant $K_2<0$. Then any holomorphic map $f$ from $M$ into $N$ satisfies
 \begin{equation*}
(f^*G)(z;dz) \leq \frac{K_1}{K_2}ds_M^2.
\end{equation*}
\end{theorem}

 \begin{theorem}\cite[p. 1676, Theorem 6.4 ]{NZ1}\label{EX}
 Suppose that $\mathcal{D}\subset \mathbb{C}^n(n\geq 2)$ is a bounded domain. Then $\mathcal{D}$ admits a non-Hermitian quadratic strongly pseudoconvex complex Finsler metric $G:T^{1,0}\mathcal{D}\rightarrow [0,+\infty)$ such that its holomorphic sectional curvature is bounded from above by a negative constant.
 \end{theorem}
\begin{remark}\label{R-6.1}
Without lose of generality, we assume that $\mathcal{D}$ contains the origin $0\in\mathbb{C}^n$. Set
 $$M_0:=\sup_{z\in \mathcal{D}}\{\|z\|\}>0,$$
here $\|\cdot\|$ denotes the canonical complex Euclidean norm of $z\in \mathcal{D}$. For any $z\in \mathcal{D}, 0 \neq v\in T_z^{1,0}\mathcal{D}=\mathbb{C}^n$, we denote
$$r:=\langle v, v \rangle=\|v\|^2,\quad t:=\langle z, z \rangle=\|z\|^2, \quad s:=\frac{|\langle z,v\rangle|^2}{r},$$
where $\langle\cdot,\cdot\rangle$ denotes the canonical complex Euclidean inner product in $\mathbb{C}^n$. For every constants $a,b$ satisfying $a>0$ and $0<b<\frac{1}{M_0}$, we define
\begin{equation*}
G(z,v):=r\phi(t,s)\quad\mbox{with}\quad \phi(t,s)=e^{at+bs},\quad \forall z\in \mathcal{D}, v\in T_z^{1,0}\mathcal{D}.\label{G}
\end{equation*}
It is clear that $G:T^{1,0}\mathcal{D}\rightarrow [0,+\infty)$ is a non-Hermitian quadratic strongly pseudoconvex complex Finsler metric such that its holomorphic sectional curvature is bounded from above by a negative constant from Theorem 6.4 in \cite{NZ1} and Proposition 2.9 in \cite{ZCP}. If $b$ is equal to $0$ and $a$ is greater than $0$, $G$ is a Hermitian metric such that its holomorphic sectional curvature is bounded from above by a negative constant.
\end{remark}
In \cite{DGZ2}, Deng et al. proved that any bounded strictly pseudoconvex domains with $C^2$-boundary in $\mathbb{C}^n$ are HHR. In \cite{KZ}, Kim and Zhang showed that all bounded convex domains in $\mathbb{C}^n$ are HHR. Combining Theorem \ref{Bound}, the above results and Remark \ref{R-3.2}, we obtain the following Lemma.

\begin{lemma}\label{L-6.1}
Suppose that $\mathcal{D}$ is a bounded strictly pseudoconvex domain  with $C^2$-boundary or bounded convex domain in $\mathbb{C}^n$. Then holomorphic sectional curvature, Ricci curvature, scalar curvature of the Bergman metric on $\mathcal{D}$ are bounded.
\end{lemma}
\begin{theorem}\label{Ohsawa} \cite[p. 238, Theorem 1]{Ohsawa} \cite[p. 548, Theorem 1]{YSK}
Every bounded pseudoconvex domain in $\mathbb{C}^n$ with a $C^1$-boundary or domain with the uniform squeezing property is complete with respect to its Bergman metric.
\end{theorem}
\begin{theorem}\label{MT}
For a suitable choice of positive constant $C$, suppose that $\mathcal{D}$ is a bounded strictly pseudoconvex domain  with $C^2$-boundary or bounded convex domain in $\mathbb{C}^n$. Suppose that $\mathcal{D}$ admits a strongly pseudoconvex complex Finsler metric $G:T^{1,0}\mathcal{D}\rightarrow [0,+\infty)$ such that its holomorphic sectional curvature is bounded from above by a negative constant $-K$. Then $H=CG+g_B$ and the Bergman metric $g_B$ are equivalent. What's more, $H=CG+g_B$ is complete.
\end{theorem}
\begin{proof}  $ \forall v \in T^{1,0}_z\mathcal{D}\backslash\{0\}\cong\mathbb{C}^n\backslash\{0\}$. By assumptions, we observed that the holomorphic sectional curvature $Sec_{\mathcal{D}}(z,v)$ along a non-zero tangent vector $v$ of the Bergman metric $g_B$ on $\mathcal{D}$ is bounded.  By Lemma \ref{L-6.1}, we have
\begin{equation}\label{EQ-6.3}
-A \leq Sec_{\mathcal{D}}(z,v) \leq B,
\end{equation}
where $A, B$ are positive constants.

Combine the equality $H=CG+g_B$ with Propositions  \ref{L-2.2}, \ref{L-2.1}, \ref{P-2.1} , we have
\begin{align}\begin{split}\label{EQ-6.4}
K_H(v)=\sup_{\varphi}\{K(\varphi^*H)(0)\}&=\sup_{\varphi} \{K(\varphi^*(CG)+\varphi^*g_B)(0)\}\\
&\leq\sup_{\varphi} \{\frac{1}{C}K(\varphi^*G)(0)+K(\varphi^*g_B)(0)\}\\
&\leq \frac{1}{C}\sup_{\varphi}\{K(\varphi^*G)(0)\}+\sup_{\varphi}\{K(\varphi^*g_B)(0)\}\\
&= \frac{1}{C}K_G(v)+Sec_{\mathcal{D}}(z,v),
\end{split}\end{align}
where $K_H(v), K_G(v)$ are the holomorphic sectional curvature along a non-zero tangent vector $v$ of the complex Finsler metrics $H, G$ respectively. Here the supremum is taken with respect to the family of all holomorphic maps $\varphi: \mathbb{D} \rightarrow M$ with $\varphi(0)=z$ and $\varphi'(0)=\lambda v$ for some $\lambda \in \mathbb{C}\backslash \{0\}$.

By the curvature condition of $G$, \eqref{EQ-6.3}, \eqref{EQ-6.4}, we obtain
\begin{equation}\label{EQ-6.5}
K_H(v) \leq -\frac{K}{C}+B.
\end{equation}
We take $0<C<\frac{K}{B}$, we know that the holomorphic sectional curvature $K_H(v)$ is bounded from above by a negative constant $-\frac{K}{C}+B$.

 Combining with Theorem \ref{sectional}, we know that the sectional curvature of the Bergman metric $g_{B}$ on $\mathcal{D}$ is also bounded. By lemma \ref{L-6.1}, we know that the holomorphic sectional curvature of the Bergman metric on $\mathcal{D}$ is bounded from below by a negative constant $-C_1$. Applying Theorem \ref{SL} to the case where $f$ is an identity holomorphic map $id$, we have
\begin{equation}\label{in}
H=id^*H \leq \frac{-C_1}{-\frac{K}{C}+B}g_{B}=C_2g_B, C_2>0.
\end{equation}
Since $H=CG+g_B$ and $C>0$, we know that
\begin{equation}\label{EQ-6.7}
g_B\leq H.
\end{equation}
From inequalities \eqref{in} and \eqref{EQ-6.7}, $H$ is equivalent to the Bergman metric $g_B$. By Theorem \ref{Ohsawa}, we know that $H$ is a complete complex Finsler metric.
\end{proof}
\begin{remark}
If $G$ comes from a Hermitian metric, this theorem still holds. Hence, if $\mathcal{D}$ admits a complete  Hermitian metric $h$ such that its holomorphic sectional curvature is bounded from above by a negative constant, $(\mathcal{D},H)$ is a complete Hermitian manifold.
\end{remark}

\begin{theorem} \cite[p. 112, Theorem 3.7.1 ]{kobayashi} \label{completeKobayashi}
Suppose that $(M,G)$ is a complete strongly pseudoconvex complex Finsler manifold whose holomorphic sectional curvature is bounded from above by a negative constant $-B<0$. Then $M$ is a complete Kobayashi-hyperbolic.
\end{theorem}
\begin{proof}
Let $\phi \in \text{Hol}(\mathbb{D},M)$ such that $\phi(0)=z$, and $\phi_*(\xi)=v$. It follows from Schwarz lemma that
 $$ \phi^*G \leq \frac{4}{B}\omega_{\mathcal{P}} \quad \text{on} \quad \mathbb{D},$$
 where $\omega_{\mathcal{P}}= \frac{1}{(1-|z|^2)^2}dz \otimes d\overline{z}$. It follows that
\begin{equation*}
G(z,v)=(\phi^*G)(0,\xi) \leq \frac{4}{B}\omega_{\mathcal{P}}(0,\xi)=\frac{4}{B}|\xi|^2_{\mathbb{C}}.
\end{equation*} \label{A}
Hence, $|\xi|^2_{\mathbb{C}} \geq \frac{B}{4}G(z,v)$. By the definition of  Kobayashi metric, we obtain
\begin{equation} \label{K}
\mathfrak{K}^2_{M}(z,v) \geq \frac{B}{4}G(z,v)>0.
\end{equation}
Therefore, $M$ is a Kobayashi-hyperbolic.

By the equality \eqref{K}, we obtain
\begin{equation}\label{EQ-6.10}
d_K \geq \sqrt{\frac{B}{4}}d_G,
\end{equation}
where $d_K$ and $d_G$ are Kobayashi distance and the integrated form of $G$ on $M$, respectively.
Let ${z_n}$ be a Cauchy sequence in $(M,d_K)$.  By the equality \eqref{EQ-6.10}, ${z_n}$ is a Cauchy sequence in $(M,d_G)$. Since $(M,d_G)$ is complete, there is a $z_0 \in M$ such that $\lim_{n \rightarrow \infty} z_n =z_0$. Then $(M, d_K)$ is complete. Therefore, we know that $M$ is a complete Kobayashi-hyperbolic.
\end{proof}
\begin{remark}
If $(M,G)$ is from a Hermitian manifold, Theorem \ref{completeKobayashi} reduces Theorem 4.11 in Chapter 4 in \cite{kobayashi2}.
\end{remark}
As a simple application of Theorems \ref{EX}, \ref{MT}, \ref{completeKobayashi} and Remark \ref{R-6.1}, we obtain the following corollary.
\begin{corollary}\label{Coro}
Suppose that $\mathcal{D}$ is a bounded strictly pseudoconvex domain  with $C^2$-boundary or bounded convex domain in $\mathbb{C}^n$. Then $\mathcal{D}$ admits a complete strongly pseudoconvex complex Finsler metric $H:T^{1,0}\mathcal{D}\rightarrow [0,+\infty)$ such that its holomorphic sectional curvature is bounded from above by a negative constant. What's more, $\mathcal{D}$ is a complete Kobayashi-hyperbolic.
\end{corollary}
\section{Uniform equivalence of Kobayashi metric}
\noindent

In this section, we prove Theorem \ref{T-1.6} of this paper. Firstly, we prove the following Theorem \ref{T-7.1}.


 \begin{theorem}\label{T-7.1}
 Suppose that $M$ is a Lempert manifold. And let $(M,h)$ be a complete K\"ahler manifold whose holomorphic sectional curvature $K_h$ satisfies $-A \leq K_h \leq -B$ for some positive constants $A$ and $B$. Then Kobayashi metric $\mathfrak{K}_M$ satisfies
 $$ C^{-1} h(z,v) \leq\mathfrak{K}_M^2(z,v) \leq Ch(z,v), \forall z \in M, \forall v \in T^{1,0}_zM.$$
 Here $C>0$ is depending only $A, B$, and independent of complex dimension $n$ of complex manifold $M$.
 \end{theorem}
 \begin{proof}
 Let $\phi \in \text{Hol}(\mathbb{D}, M)$ such that $\phi(0)=z$, and $\phi_*(\xi)=v$. It follows from Schwarz lemma that
 $$ \phi^*h \leq \frac{4}{B}\omega_{\mathcal{P}} \quad \text{on} \quad \mathbb{D},$$
 where $\omega_{\mathcal{P}}= \frac{1}{(1-|z|^2)^2}dz \otimes d\overline{z}$. It follows that
\begin{equation*}
h(z,v)=(\phi^*h)(0,\xi) \leq \frac{4}{B}\omega_{\mathcal{P}}(0,\xi)=\frac{4}{B}|\xi|^2_{\mathbb{C}}.
\end{equation*}
Hence, $|\xi|^2_{\mathbb{C}} \geq \frac{B}{4}h(z,v)$. By the definition of  Kobayashi metric, we obtain
\begin{equation}\label{AA}
\mathfrak{K}^2_{M}(z,v) \geq \frac{B}{4}h(z,v).
\end{equation}
\par Similar to the proof of Theorem \ref{MT}, by Theorem \ref{sectional} and applying Theorem \ref{SL} to the case where a holomorphic map $f:M \rightarrow M$ is an identity map $id$, we have
\begin{equation}\label{B}
\mathfrak{K}^2_{M}(z,v)=id^*\mathfrak{K}^2_{M}(z,v) \leq \frac{A}{4}h(z,v).
\end{equation}
By the inequalities \eqref{AA}, \eqref{B}, there exists a positive constant $C$ depending only $A, B$ such that
$$ C^{-1} h(z,v) \leq \mathfrak{K}^2_{M}(z,v) \leq C h(z,v).$$
\end{proof}

Combining Theorems \ref{Lempert}, \ref{T-5.2} and Theorem \ref{T-7.1}, we prove the following corollary.
\begin{corollary}\label{C-7.1}
 Suppose that $\mathcal{D}$ is a bounded strongly convex domain with smooth  boundary in $\mathbb{C}^n$. And let $(\mathcal{D},h)$ be a complete K\"ahler manifold whose holomorphic sectional curvature $K_h$ satisfies $-A \leq K_h \leq -B$ for some positive constants $A$ and $B$. Then Kobayashi metric and Carath\'eodory metric satisfies
 $$ C^{-1} h(z,v) \leq\mathfrak{K}^2(z,v)=\mathfrak{C}^2(z,v) \leq Ch(z,v), \forall z \in M, \forall v \in T^{1,0}_zM.$$
 Here $C>0$ is only depending $A, B$, and independent of complex dimension $n$ of complex manifold $\mathcal{D}$.
 \end{corollary}

Now we complete the proof of Theorem \ref{T-1.7}.
\begin{theorem}\label{T-7.2}
For a suitable choice of positive constant $C$, suppose that $(M,h)$ is a complete K\"ahler manifold whose holomorphic sectional curvature $K_h$ satisfies $-A \leq K_h \leq B$ for some positive constants $A$ and $B$. If Carath\'eodory metric $\mathfrak{C}^2$ on complex manifold $M$ is a smooth  strongly pseudoconvex complex Finsler metric. Then $H=C\mathfrak{C}^2+h$ and the K\"ahler metric $h$ are equivalent. What's more, $M$ is a complete Kobayashi-hyperbolic manifold.
\end{theorem}
\begin{proof}
By Wong's result \cite{WB}, we know that the holomorphic sectional curvature of the Carath\'eodory metric is less than $-4$. In the similar way as Theorem \ref{MT}, we know that the holomorphic sectional curvature of $H$ is bounded above by a negative constant $-\frac{4}{C}+B(0<C< \frac{4}{B})$. By assumptions and Theorem \ref{sectional},  applying Theorem \ref{SL} to the case where a holomorphic map $f: M \rightarrow M$ is an identity map $id$, we have
\begin{equation*}
H=id^*H\leq \frac{-A}{-\frac{4}{C}+B} h= C_3h, C_3>0.
\end{equation*}
Since $H=C\mathfrak{C}^2+h$ and $C>0$, we have
\begin{equation*}
H \geq h.
\end{equation*}
So $H=C\mathfrak{C}^2+h$ is equivalent to $h$. Because of the completeness of $(M,h)$, we get that $(M,H)$ is a complete complex Finsler manifold. By Theorem \ref{completeKobayashi}, $M$ is a complete Kobayashi-hyperbolic.
\end{proof}

{\bf Acknowledgement:} {
The author thanks the referees for carefully reading the manuscript and their valuable corrections/suggestions which improved the presentation of the paper. This work was supported by National Natural Science Foundation of China (Grant No. 12401101, No. 12461008, No. 12461013) and Natural Science Foundation of Jiangxi Province in China (Grant No. 20232ACB201005).}

\end{document}